\documentclass{amsart}
\usepackage{amssymb}
\usepackage{color}
\usepackage{graphicx}
\usepackage{amsmath,amscd}
\usepackage{psfrag}

\newtheorem{proposition}{Proposition}
\newtheorem{theorem}[proposition]{Theorem}
\newtheorem{definition}[proposition]{Definition}
\newtheorem{lemma}[proposition]{Lemma}

\newtheorem{corollary}[proposition]{Corollary}
\newtheorem{remark}[proposition]{Remark}

\newtheorem{example}[proposition]{Examples}

\begin{document}

\title{Embeddings of $3$--manifolds via open books}

\date{\today}


\author{Dishant M. Pancholi}
\address{Institute of Mathematical Sciences,
IV, Cross Road, CIT Campus, 
Taramani, 
Chennai 600133,
Tamilnadu, India.}
\email{dishant@imsc.res.in}

\author{Suhas Pandit}
\address{Indian Institute of Technology Madras,
IIT PO.Chennai, 600036, Tamilnadu, India. }
\email{suhas@iitm.ac.in}

\author{Kuldeep Saha}
\address{Chennai Mathematical Institute,
H1, SIPCOT IT 	Park, Siruseri, Kelambakkam
603103, Tamilnadu, India}
\email{kuldeep@cmi.ac.in}



\begin{abstract}
In this note, we discuss embeddings of $3$--manifolds via open books. 
First we show that every open book of every closed orientable $3$--manifold admits an open book
embedding in any open book decompistion of $S^2 \times S^3$ and $S^2 \widetilde{\times} S^3$ with
the page a disk bundle over $S^2$ and monodromy the identity.  We  then
use  open book embeddings to reprove  that every
closed orientable $3$--manifold embeds in $S^5.$
\end{abstract}
\maketitle
\section{Introduction}\label{sec:intro}

Let $M$ be a closed oriented smooth manifold and $B$ be a co-dimension $2$ oriented smooth submanifold with a trivial 
normal bundle in $M$. We say that $M$ has an \emph{open book decomposition}, denoted by $\mathcal{O}b(B,\pi)$, if  
$M \setminus B$ is a locally trivial fibre bundle over $S^1$ such that the fibration $\pi$ in a neighborhood of 
$B$ looks like the trivial fibration of $(B \times D^2) \setminus (B \times \{0\}) \rightarrow S^1$ sending $(x, r, \theta)$ 
to $\theta,$  where $x \in B$ and $(r,\theta)$ are polar co-oridinates on $D^2$ and the boundary of each fibre is $B$ 
which we call the binding. The closure of each fibre is called the page of the open book and the monodromy of the 
fibration is called the monodromy of the open book. An open book decomposition of $M$ is determined -- up to diffeomorphism 
of $M$ -- by the topological type of the page $\Sigma$ and the isotopy class of the monodromy which is an element of the 
mapping class group of $\Sigma$.

In \cite{Al}, J. Alexander proved that every closed oriented $3$--manifold admits an open book.  Open book decompositions of 
closed oriented simply connected manifolds were studied by H. Winkelnkemper in \cite{Wi}, where  he proved the existence of 
such decompositions on closed oriented simply connected manifolds of dimension $n$ at least $6$, provided $n$ is not divisible 
by $4.$ He also established that if the  dimension $n>6$ of a closed simply connected manifold is divisible by $4$, then 
it admits an open book decomposition if and only if its signature  is zero.   Winkelnkemper's results were then extended by   
J. Lawson \cite{La}, F.Quinn \cite{Qu} and I. Tamura \cite{Ta}. Due to their works,  the conditions under which a manifold 
admits an open book decomposition is now well known. These conditions are generally very mild and hence a very large class 
of manifolds satisfy them. A particular class of manifolds that will be of interest to us consists of odd dimensional closed 
orientable manifolds. It can be easily deduced from \cite{Wi}  that every closed orientable odd dimensional manifold admits 
an open book decomposition. See also \cite{Qu} for more details.

In recent times, the study of open book decompositions of manifolds has become very prominent due to the connection -- 
discovered by E. Giroux \cite{Gi} -- between the open book decompositions and the contact structures. Let $M$ be an odd 
dimensional smooth manifold. Recall that a contact structure $\xi$ on $M$  is a nowhere integrable co-dimension one 
distribution. Giroux in his seminal work \cite{Gi} showed that there is a one to one correspondence between the isotopy 
classes of co-oriented contact structures on a closed oriented $3$--manifold $M$  and the open book decompositions of $M$ up 
to  positive stabilizations. By a positive stabilization operation on an open book of $M$, we mean just the plumbing 
of a positive Hopf band to the page of the open book. More precisely, it is adding a $1$--handle to the page and 
modifying the monodromy by adding a positive Dehn twist along a curve which goes over the handle exactly once. 
For more details on this, see \cite{Et}.
See also,   \cite{Gi} and \cite{Ko} for more regarding Giroux's correspondence.

In this article, we first study open book embeddings of $3$--manifolds. We say that a smooth manifold $M$ with a given open
book decomposition  admits an open book embedding in an open book decompistion of a  smooth manifold $N$, provided there is 
an embedding of $M$ in $N$ such that --as a submanifold of $M$-- the given open book decomposition on $M$ is compatible with 
the open book  decomposition of $N$. Such a submanifold as an open book  is also known as a nested open book of $N$ with the 
given open book, see \cite{DK}. A more precise definition of an open book embedding is given 
in Section~\ref{sec:section-3}. 

Open book embeddings were used by  A. Mori \cite{Mr} to prove that every closed co-oriented contact $3$-manifold open book 
embeds in $S^7.$ His result was generalized by D. Martinez Torres in \cite{Mr} to show that there is
an open book embedding of any contact manifold of dimension $2n+1$ in $S^{4n+3}.$ In addition,  the article 
by J. Etnyre and Y. Lekili~\cite{EL}  uses open book embeddings to produce contact embeddings. In particular,
they establish that there exists a  contact structure $\xi_{ot}$ on $S^2 \times S^3$ in which every co-orinentable 
contact $3$--manifold embeds in such a way that the pull-back of $\xi_{ot}$ via this embedding induces the given
contact structure on the contact $3$--manifold.  We would like to mention that  it is not
known, if every $3$--manifold admits an open book embedding in $S^5.$ 
Our first theorem is regarding open book embeddings of  $3$--manifolds in $S^3 \times S^2$ and $S^2 \widetilde{\times} S^3$.  
Before we state this result we would like to mention that in this article, 
we work in smooth category, i.e.,  all maps and manifolds if not stated otherwise are smooth. 
We now state the theorem:

\begin{theorem}\label{thm:smooth_ob_embedding}
Let $M$ be a closed  oriented connected  $3$--dimensional manifold together with an open book decomposition 
$\mathcal{O}b(B, \pi).$ Then,  open book $\mathcal{O}b(B, \pi)$ admits an open book embedding in any open book 
decomposition associated to $S^3 \times S^2$ with pages a disk bundle over $S^2$ of even 
Euler number and monodromy the identity as well as in any open book of  $S^3 \widetilde{\times} S^2$  with 
pages a disk bundle over $S^2$ of odd Euler number and 
monodromy the identity.
\end{theorem}

Using methods used in establishing the Theorem~\ref{thm:smooth_ob_embedding}, we  establish the following:

\begin{theorem}\label{thm:smooth_embedding_in_S^5}
Every closed  orientable $3$--manifold admits a smooth embedding in $S^5.$ 
\end{theorem}

This theorem  was first discovered by M. Hirsch in \cite{Hi}. Embeddings of manifolds in Euclidean spaces have a long 
history starting from the seminal work of H.Whitney \cite{Wh} establishing that every closed $n$--manifold admits 
an embedding in $\mathbb{R}^{2n}.$ In fact, a general result of A. Heafliger and M.Hirsch \cite{HH} implies that every odd dimensional closed 
orientable manifold embeds in $\mathbb{R}^{2n-1}.$
There are other  proofs of  Theorem~\ref{thm:smooth_embedding_in_S^5}. See, for example, the article \cite{HLM} 
for a proof using what is now known as braided embeddings and also the article \cite{Ka} for
embeddings of closed orientable $3$--manifolds in $S^5$ using surgery description of $3$-manifolds
and Kirby calculus. We refer to \cite{Wa} and 
\cite{Ro}  for  embeddings of non-orientable  $3$--manifolds in $\mathbb{R}^5.$

\subsection{Acknowledgement}  We are  thankful to John Etnyre for various comments that has helped us to
improve the presentation of this article. The first author is thankful to Simons Foundation for providing
support to travel to Stanford, where a part of work of this project was carried out. The first author
is also thankful to ICTP, Trieste, Italy and Simons Associateship program without which this work
would not have been possible. Finally, we would like to thank Yakov Eliashberg for asking various 
questions regarding embeddings of contact manifolds that stimulated this work. We are also very thankful to 
the anonymous referee for crictical comments and suggestions.

\section {Preliminaries}\label{sec:prelim}

In this section, we quickly review notions necessary for this article  pertaining mapping class groups and open book 
decompositions. 

\subsection{Mapping class group}
\mbox{}

Let us begin by recalling the definition of a mapping class group as in \cite{FM}.

\begin{definition}[Mapping class group]
Let $\Sigma$ be an orientable manifold. By the mapping class group of $\Sigma,$ we mean
the group of orientation preserving self diffeomorphisms of $\Sigma$ upto isotopy. In case, $\Sigma$ has a non-empty 
boundary $\partial \Sigma$, then we always assume  that diffeomorphisms and the isotopies are  the identity in a collar 
neighborhood  of the boundary.
\end{definition}

We denote the mapping class group of a surface $\Sigma$ by $\mathcal{M}CG(\Sigma).$ In case, $\Sigma$ has a non-empty
boundary and we want to emphasis this fact, we will denote the mapping class group by $\mathcal{M}CG(\Sigma, \partial \Sigma).$
In this article,  unless  specified otherwise, we say that two diffeomorphisms $f$ and $g$ of a manifold $(M, \partial M)$ 
are equal provided they represent the  same element in $\mathcal{M}CG(M, \partial M).$

Lickorish in \cite{Li} showed that every element of the mapping class group of an orientable surface is a product of Dehn twists.

 Recall that, by definition, a  Dehn twist along the circle $S^1\times \{1\}$  in $S^1\times [0,2]$ is a self diffeomorphism of  $S^1 \times [0,2]$ 
given by $(e^{i\theta}, t)$ going to $(e^{i(\theta \pm \pi t)}, t).$ Note that the Dehn twist fixes both the boundary components of $S^1 \times [0,2].$ Hence, given  an embedded circle $c$ in an orientable  surface $\Sigma,$ we can define the Dehn twist
of an annular neighborhood $N(c)=S^1\times[0,2]$ of $c=S^1\times \{1\}$ in $\Sigma$ which is the identity when restricted to the boundary of this annular neighborhood. Clearly, we can extend this diffeomorphism by the  identity in the complement of the annular neighborhood to produce a self diffeomorphism of $\Sigma$. This diffeomorphism is called a Dehn twist along the embedded cirlcle $c$ in the surface $\Sigma$. For more details,  refer  \cite{FM}.

 \begin{figure}[ht]
\begin{center}
\psfrag{A}{$\Sigma$}
\psfrag{B}{$\partial \Sigma$}
\includegraphics[width=10cm,height=5cm]{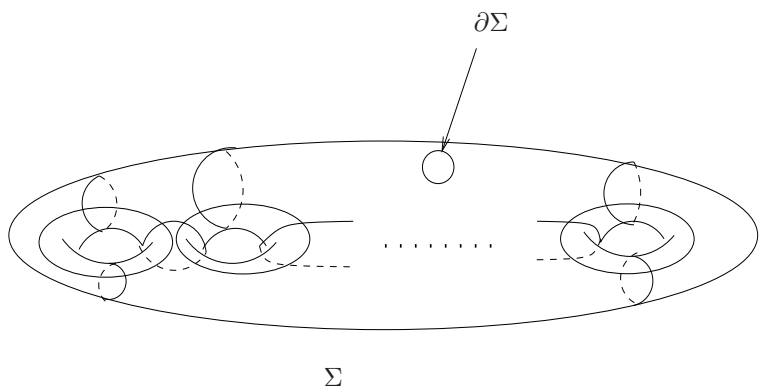}
\caption{Figure depicts  genus $g$ compact orientable surface $\Sigma$ with one boundary component. The embedded 
 curves on the surface represents the standard Lickorish generators corresponding to 
 the presentation of the mapping class group of  $\Sigma$  as given in \cite{Jo}.}
 \label{fig:lic_gen}
  \end{center}
 \end{figure}







In fact, it was later established by Lickorish that every element of the mapping class group of a closed orientable surface
of genus $g$ is a product of Dehn twists along curves depicted in the Figure~\ref{fig:lic_gen}.
For an orientable surface with connected boundary, it was established by D.L. Johnson~\cite{Jo} that the same
generators that generate the mapping class group of the closed surface obtained by attaching a disk 
along the boundary of the  surface are sufficient to generate the mapping class group.  In case, the boundary
of the surface has more than one connected components, then we need -- additionally -- Dehn twists along certain simple
closed curves which separate boundary components. See \cite[p.~133]{FM} for a picture depicting generators in this case.

 Let $\mathcal{C}$ be the  finite collection of simple closed curve embedded on a surface $\Sigma$ as
 depicted in \cite[p.~133]{FM}.
 We know that  Dehn twists along the curves in $\mathcal{C}$  generate the mapping class group of $\Sigma.$
 Any curve $c \in \mathcal{C}$ will be refered as 
 a \emph{Lickorish curve} and Dehn twists along these Lickorish curves  as \emph{Lickorish generators}.

\subsection{Open books}
\mbox{}

Let us review few results related to open book decompositions of manifolds.
We first recall the following:

\begin{definition}[Open book decomposition] 
An open book decomposition of a closed oriented  manifold $M$ consists of a co-dimension $2$ oriented submanifold $B$ with a trivial normal bundle in $M$ and a  locally trivial fibration
$ \pi: M \setminus B \rightarrow S^1$ such that $\pi^{-1}(\theta)$ is an interior of a
co-dimension $1$ submanifold $N_{\theta}$ and $\partial N_{\theta}  = B$, for all $\theta.$
The submanifold $B$ is called the binding and  $N_{\theta}$ is called a page 
of the open book. We denote the open book decomposition of $M$ by $(M, \mathcal{O}b(B, \pi))$ or sometimes
simply by $\mathcal{O}b(B, \pi).$
\end{definition}

Next, we discuss the notion of an \emph{abstract open book decomposition}. To begin with, let us recall
the following:

\begin{definition}[Mapping torus] Let $\Sigma$ be a manifold with non-empty boundary $\partial \Sigma$. 
Let $\phi$ be an element of the mapping class group of $\Sigma$. By the mapping torus 
$\mathcal{MT}(\Sigma, \phi),$ we mean

$$ \Sigma \times [0,1] / \sim $$  

where $\sim$  is the equivalence relation identifying $(x, 0)$ with $(\phi(x), 1).$
\end{definition}

We are now in a position to define an abstract open book decomposition.

\begin{definition}\label{def:ab_open_book}
Let $\Sigma$ and $\phi$ as in the previous definition. An abstract open book decomposition of 
$M$ is pair $(\Sigma, \phi)$ such that $M$ is diffeomorphic to 

$$\mathcal{MT}(\Sigma, \phi) \cup_{id} \partial \Sigma \times D^2 $$

where $id$ denotes the identity mapping of $\partial \Sigma \times S^1.$
\end{definition}

Note that the mapping class $\phi$ determines $M$ uniquely up to diffeomorphism. The map $\phi$ is called the \emph{monodromy} of the open book. 
The manifold obtained by identifying the boundary of $\mathcal{M}T(\Sigma, \phi)$ with the boundary
of $\partial \Sigma \times D^2$ as described  in the Definition~\ref{def:ab_open_book} will
be denoted by $\mathcal{A}ob(\Sigma, \phi).$ A manifold $M$ together with a given  abstract open book decomposition  will be denoted by $(M,\mathcal{A}ob(\Sigma, \phi)).$ 

One can easily see that an abstract open book decomposition of $M$ gives an open book decomposition of $M$
up to diffeomorphism and vice versa. Hence, we will not generally distinguish between open books and
abstract open books.

\begin{remark}\label{rmk:standard_ob}
\begin{enumerate}
\item
Notice that $S^n$ admits an open book decomposition with pages $D^{n-1}$ and the monodromy the mapping class $Id$ of the identity map of $D^{n-1}.$ We call this open book the \emph{trivial open book}.
For more details regarding open books,  refer  the lecture notes \cite{Et} and \cite{Gi}[chpt-4.4.2]. 

\item $S^3 \times S^2$ admits an open book decomposition with pages the unit disk bundle of  $T^*S^2$ and  monodromy  the mapping class of the identity map of the unit disk bundle of $T^*S^2$. We call this open book decomposition of $S^3 \times S^2$  the standard open book decomposition of $S^3 \times S^2.$

\end{enumerate}

\end{remark}

\section{Open book embeddings of $3$-manifolds in $S^3 \times S^2$ and $S^2 \widetilde{\times} S^3$}
\label{sec:top_ob_embedding}

In this section, we produce open book embeddings of closed oriented $3$-manifolds in any  open books of 
$S^3 \times S^2$ and  $S^2 \widetilde{\times} S^3$ with pages any disk bundle and monodromy the identity.
We begin by reviewing   quickly  some well known results about embedded Hopf band in $S^3.$  
We can view $S^3$ as the unit sphere in $\mathbb C^2$. The Hopf links $H^{\pm}$ are the pre-images of $0$ under 
the maps $(z_1,z_2)\to z_1z_2$ and $(z_1,z_2)\to z_1\bar{z_2}$, respectively restricted to the 
unit sphere $S^3$ of $\mathbb{C}^2.$ A Hopf annulus is a Seifert surface for a 
Hopf link and a positive/negative Hopf band in $S^3$ is an embedded annulus with the boundary $H^{\pm}$. See, for example, 
\cite{Et} for more details.

\subsection{Hopf band in $S^3$ and the mapping class group of an annulus}\label{sec:hopf_annulus}
\mbox{} 
 
 To begin with, we go through the proofs of the following well known results. These are also proved in \cite{HY}.

\begin{lemma}\label{lem:annulus_in_sphere}
 Let $A$ be an annulus and let $\phi$ be an element of the mapping 
class group $\mathcal{M}CG(A)$ of $A$. Then, there exists an embedding $f$ of $A$ in $S^3$ that
satisfies the following:
\begin{enumerate}
\item $f(A)$ is a Hopf band in $S^3.$
\item There exists a diffeomorphism of $\Psi_1$ of $S^3$, isotopic to the identity via
an isotopy $\Psi_t$ such
that $f^{-1} \circ \Psi_1 \circ f = \phi.$
\item The isotopy $\Psi_t$ fixes the boundary of $A$ pointwise for all $t.$ 

\end{enumerate}

\end{lemma}

\begin{proof}
We know that $S^3$ admits an open book decomposition with pages a Hopf band and the monodromy the  Dehn twist around its center  circle. This, in particular, implies that there exists a flow $\Phi_t$ on $S^3$  whose time $1$ map  $\Phi_1$ maps  a Hopf annulus -- say $\mathcal{A}$  -- to itself and $\Phi_1$ restricted to $\mathcal{A}$ is a Dehn twist along the center circle on $\mathcal{A}.$ We consider an embedding $f$ of $A$ in $S^3$ such that $f(A)=\mathcal A$. The lemma is now a straight forward consequence of the fact that every element of the mapping class group of an annulus is just a power of the Dehn twist along its center circle.  
\end{proof}

Note that $S^3 \times [0,1]$   can be regarded as a collar of $\partial D^4$ in $D^4$ with $\partial D^4=S^3\times 1$. Since 
$\Psi$ constructed in the Lemma~\ref{lem:annulus_in_sphere} is isotopic to the identity, we 
have the following: 


\begin{corollary}\label{cor:mcg_annulus}
There exists a proper embedding $f$ of an annulus $A$ in $(D^4, \partial D^4)$ which satisfies the property that
for every element $\phi \in \mathcal{M}CG(A)$, there exists a diffeomorphism $\Gamma_1$ of $(D^4, \partial D^4)$
isotopic to the identity such that $ \phi = f^{-1} \circ \Gamma_1 \circ f.$

\end{corollary}


\begin{proof}
First, we consider a proper embedding of $A$ in $S^3\times [0,1]$ as follows:  We smoothly push a Hopf annulus, say  $\mathcal A$ from $\partial D^4= S^3\times \{1\}$ to the level $S^3\times \{0\}$ keeping the boundary of the Hopf annulus fixed such that $S^3\times \{t\} \cap \mathcal A$ is a Hopf link for each $t\in (0,1]$. We consider the proper  embedding  $f$ of $A$ such that image of $f$ is the pushed Hopf annulus $\mathcal A$.

\noindent Now, let $\Psi_t$ be the isotopy of $S^3$ such that $\Psi_1$ realizes
the given element of $\mathcal{M}CG(A).$  Using the isotopy $\Psi_t$, we 
construct a diffeomorphism $\Gamma_1$ 
of $S^3 \times  [-1,1]$ that satisfies the following:
\begin{enumerate}

\item $\Gamma_1$ is isotopic to the identity via a family of diffeomorphisms $\Gamma_t.$ 

\item $\Gamma_1$ restricted to $S^3 \times \{0\}$ is $\Psi_1.$

\end{enumerate}

\noindent This diffeomorphism is defined as follows:

$$
\Gamma_1 (x,t) 
 =
\left\{
	\begin{array}{ll}
	\Psi_{1-t}(x)  & \mbox{if } t \geq 0 \\
         \Psi_{t+ 1}(x) & \mbox{if } t \leq 0
	\end{array}
\right . 
$$

\noindent Since $S^3 \times [-1,1]$ can be regarded as a collar of $\partial D^4$ in $(D^4, \partial D^4)$, 
we are through as $\Gamma_1$ clearly can be extended smoothly to a diffeomorphism of $(D^4, \partial D^4)$ by the 
identity in the complement of the collar.
\end{proof}

\subsection{Open book embeddings}\label{sec:section-3}
\mbox{}

In this section, we review the notion of  open book embeddings. More concretely, we will make precise the notion
of abstract open book embeddings.  However, as earlier remarked, since open book decomposition and abstract
open book decompositions are closely related, we will often not distinguish between abstract open book 
embeddings and open book embeddings.

\begin{definition}[Open book  embedding]
Let $M^k$ be a manifold with open book decomposition $\mathcal{O}b(B_1, \pi_1)$ and  $N^l$ be another manifold
with open book decomposition $\mathcal{O}b(B_2, \pi_2).$ We say an embedding $f: M \hookrightarrow N$ is an
open book embedding of $(M, \mathcal{O}b(B_1, \pi_1))$  in $(N, \mathcal{O}b(B_1, \pi_2))$
provided $f$ embeds $B_1$ in $B_2$ and the following diagram commutes:

\begin{equation*}
\begin{CD}
M \setminus B_1   @>f>>  N \setminus B_2\\
@VV\pi_1V        @VV\pi_2V\\
S^1     @>id>>  S^1
    \end{CD}
\end{equation*}

\end{definition}

 Just as an abstract open book is defined, we can define an abstract open book embedding as follows:

\begin{definition}[Abstract open book embeddings]
Let $M = \mathcal{A}ob \left(\Sigma_1, \phi_1\right)$ and $N = \mathcal{A}ob\left( \Sigma_2, \phi_2 \right)$ be
two abstract open books.  We say that there exists an abstract open book embedding of $M$ in $N$ 
provided there exists a proper embedding  $f$ of $\Sigma_1$ in $\Sigma_2$ such that $\phi_2  = f^{-1} \circ
\phi_1 \circ f.$
\end{definition}

It is clear from the definition that an abstract open book embedding produces an embedding for the associated 
open book and vice versa.


\begin{remark}
 The open book embeddings defined above are also known as spun embeddings in the litreture. 
\end{remark}

There are some obvious examples of open book embeddings. 

\begin{example}

\begin{enumerate}

\item Each sphere $S^n$ embeds in $S^{n+k}$ with $k >0$ via obvious inclusion such that the trivial open
book of $S^{n+k}$ restricts to the trivial open book of $S^n.$

\item Notice that since we can embed $S^{n-1} \times I$ in $D^{n+1}$ properly, $S^1 \times S^n$ admits 
an open book embedding in $S^{n+2}.$

\end{enumerate}
\end{example}

Now, we have an easy consequence using the Corollory  \ref{cor:mcg_annulus}.

\begin{proposition}
Any closed oriented $3$--manifold  with an open book decomposition having 
pages an annulus $A$ and the monodromy any mapping class $\phi$ of the annulus admits an open book embedding in the 
trivial open book of $S^5.$ 
\end{proposition}
\begin{proof}
 The corollory  \ref{cor:mcg_annulus} implies that the abstract open book $\mathcal{A}ob \left(A, \phi\right)$ 
 asociated to $M$ abstract open book embeds in the abstract open book $\mathcal{A}ob \left(D^4,Id\right)$ associated to $S^5$. 
 Hence, the result follows.  
\end{proof}

\subsection{The proof of the Theorem \ref{thm:smooth_ob_embedding}}
\mbox{}

In this subsection, we establish the Theorem~\ref{thm:smooth_ob_embedding}. Recall that we need to show that every 
closed oriented $3$--dimensional manifold with a given open book decomposition  open book embeds in any open book 
$S^3 \times S^2$ and $S^2 \widetilde{\times} S^3$ having pages a disk bundle over $S^2$ and its monodromy  the identity. 

We begin by introducing few terminologies. We refer to \cite{GS} for more details regarding these. We know that when we add a $2$--handle to a $4$--ball $B^4$ along an unknot on the boundary with framing 
$m, m \in \mathbb{Z}$  we produce a disk bundle with Euler number $m.$ 
Let us denote this disk bundle by $\mathcal{D}E(m).$ 

Next, we establish  a lemma. The techniques used in the proof of this lemma is adopted from techniques developed by Hirose and Yasuhara in \cite{HY} to establish \emph{flexible} embeddings of closed surfaces in certain $4$--manifolds.
Hirose and Yasuhara called an embedding $f$ of a surface $\Sigma$ in a $4$--manifold $M$ flexible provided  for every element $\phi$ of the mapping class group of $\Sigma,$ there exists a diffeomorphism $\Psi$ of $M$,
isotopic to the identity, which maps $f (\Sigma)$ to itself and $f^{-1} \circ \Psi|_{f(\Sigma)} \circ f = \phi.$

\begin{lemma}\label{lem:flexible_embedding}
Let $(\Sigma, \partial \Sigma) $ be a surface with non-empty boundary. There exists an embedding $f$ of $\Sigma$ in a disk bundle $\mathcal{D}E(m),$ for any $n \in \mathbb Z$, which satisfies the following:

\begin{enumerate}
\item The embedding is proper.
\item Given any diffeomorphism $\phi$ of $(\Sigma, \partial \Sigma),$ there exists a family $\Psi_t$ of 
diffeomorphisms   of $\mathcal{D}E(m)$ with $\Psi_0 = id$ such that $\Psi_1$ maps $\Sigma$ to itself and 
satisfies the property that $ f^{-1} \circ \Psi_1 \circ f $ is isotopic to the given diffeomorphism 
 $\phi$ of $(\Sigma, \partial \Sigma).$ 
\end{enumerate}
\end{lemma}

 
 

%
\begin{figure}[ht]
\begin{center}
\psfrag{S}{$\Sigma$}
\psfrag{D1}{$D_1$}
\psfrag{D2}{$D_2$}
\psfrag{Dn}{$D_n$}
\includegraphics[width=8.5cm,height=3.5cm]{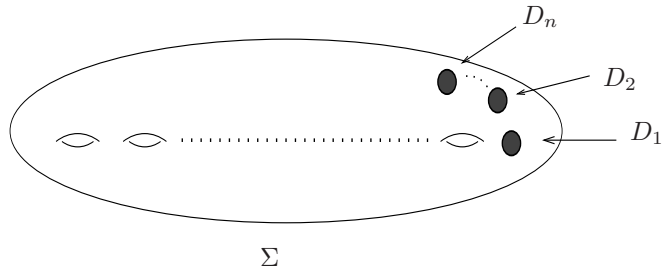}
\caption{Embedding of $\Sigma$ together with disks $D_1, \cdots ,D_n$}
\label{fig:handle_body_1}
\end{center}
\end{figure}

\begin{proof}
We know that $\mathcal{D}E(m)$ is obtained by attaching a $2$-handle to $B^4$ along an unknot with  its framing $m$. 
This implies that
we can regard it as a union of  $B^4$ with $D^2 \times D^2.$
We first describe an embedding of  $(\Sigma, \partial \Sigma)$ in $S^3 = \partial B^4$ that
we will need in order to establish the Lemma. Let us assume that $\partial \Sigma$ has
$n \in \mathbb{N}$ boundary components.  Let us denote by $\Sigma$ the
closed surface obtained from $(\Sigma, \partial \Sigma)$ after attaching disks to each boundary 
component of $\partial \Sigma.$  First, embed $\Sigma$ in $S^3$ such that it bounds the standard unknotted handle-body as shown in the Figure~\ref{fig:lic_gen}.

 Now,  observe that by removing the disks $D_1, D_2, \cdots, D_n$ as  shown in Figure~\ref{fig:handle_body_1}, 
we get an  embedding of $(\Sigma, \partial \Sigma)$ in $S^3$ such that each boundary component of 
 $(\Sigma, \partial \Sigma)$ is the boundary of  $D_i$ for some $i.$

 Next, we  attach a band with one full-twist around a properly embedded arc  in  the disk $D_1$ to the surface $\Sigma$ as shown in 
Figure~\ref{fig:hopf_band_surface}. This produces an embedded surface $S$  with $(n+1)$ boundary 
components in $S^3.$ Notice that out of these $n+1$ boundary components, $n-1$ boundary components
correspond to boundaries of the disks $D_i, i = 2,\cdots,n$. The remaining two 
boundary components form a Hopf link as depicted in
Figure~\ref{fig:hopf_band_surface}. We denote these boundary components by $H_1$ and $H_2$.
We use this embedding of the surface $S$ with $n+1$ boundary components  to properly
embed the surface $(\Sigma, \partial \Sigma)$ in $\mathcal{D}E(m)$ in the following way:

\vspace{0.25cm}

\begin{figure}[ht]
\begin{center}
\psfrag{S}{$\Sigma$}
\psfrag{D1}{$D_1$}
\psfrag{D2}{$D_2$}
\psfrag{Dn}{$D_n$}
\includegraphics[width=8.5cm,height=4cm]{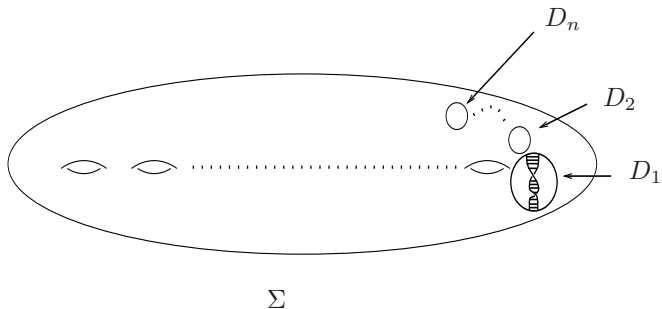}
\caption{Embedding of the surface $S$ with $n+1$ boundary components which contains a
  Hopf band $H$ as a subsurface}
  \label{fig:hopf_band_surface}
  \end{center}
 \end{figure}
\begin{figure}[ht]
\begin{center}
\psfrag{C}{$\mathcal{D}$}
\psfrag{A}{$L$}
\psfrag{B}{$C_H \#_b L$}
\psfrag{D1}{$D_1$}
\psfrag{D2}{$D_2$}
\psfrag{Dn}{$D_n$}
\includegraphics[width=7.5cm,height=4cm]{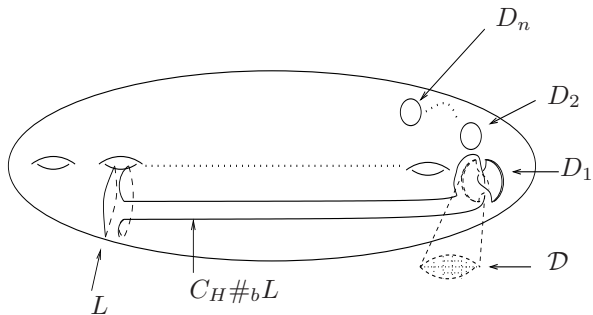}
\caption{Embedding of the surface $S$ with $n+1$ boundary components. The boundary component with
  dashed line bounds a properly embedded  disk in $\mathcal{D}E(m)$}
\label{fig:hopf_band_disk}
\end{center}
\end{figure}

 Observe that by construction,  $S$ admits an embedding of a Hopf band $H$ with the boundary components $H_1$ and $H_2$
 as shown in Figure~\ref{fig:hopf_band_surface}. 
Now, consider $S^3$ being embedded  as $S^3 \times \{\frac{1}{2}\}$ in $S^3 \times [0,1]$, where 
we regard $S^3 \times [0,1]$ as a collar of $\partial B^4.$ We now observe that
we can attach a $2$-handle along one of the boundary components of the Hopf band in such a way that we
obtain $\mathcal{D}E(m)$ from $B^4$.  More precisely, consider one of the
boundary components -- say $H_1$ --  of the Hopf band and consider the cylinder $H_1 \times [\frac{1}{2},1]$ 
and assume that $H_1 \times \{1\}$ is the unknot along which the $2$--handle with framing $n$ is attached. 
In Figure~\ref{fig:hopf_band_disk},  the boundary component $H_1$ is denoted by  a dashed  circle. 
Thus, $H_1$ bounds a disk $D$ in  
$\mathcal{D}E(m)$. We attach this disk to the surface $S$ to get a new embedding -- say $\widetilde{f}$ -- of 
$(\Sigma, \partial \Sigma)$ in $\mathcal{D}E(m)$  with its  $n$ boundary components. 
Let us denote these boundary components   by
   $\partial D_1, \partial D_2, \cdots , \partial D_n$  as shown  in Figure~\ref{fig:hopf_band_disk}.

 Consider $n$ cylinders  $\partial D_i \times [\frac{1}{2}, 1]$ for $i=1$ to $n$.  
Using these cylinders, we now modify the embedding $\widetilde{f}$ to get  a proper embedding $f$ of $(\Sigma, \partial
\Sigma)$ in $\mathcal{D}E(m).$ This we do  by considering the union 
$\widetilde{f}\left( \Sigma \right) \cup \partial D_1 \times [\frac{1}{2},1] \cup
\cdots \cup \partial D_n \times [\frac{1}{2}, 1].$

The embedding described above then clearly gives a proper embedding of $(\Sigma, \partial \Sigma)$ in $\mathcal{D}E(m)$. We
perturb this embedding -- if necessary -- to  make it into a smooth and proper embedding. By slight abuse of notation,  
let us again denote this embedding of $(\Sigma, \partial \Sigma)$ by $f.$

%
%
%

 We now observe that the embedding $f$ satisfies the property that any  simple closed curve $C$ and 
its ambient band connected sum with the center curve $C_H$ (depicted by dark cure in the Figure~\ref{fig:hopf_band_disk}) 
of the Hopf band $H$,  are ambiently isotopic. This is because, $C_H$ is isotopic  the boundary component of $H_1$
which bounds the disk $D.$ Hence, $C_H$ can be shrunk to a point in the interior of $f(\Sigma).$
This implies that  we can  isotope $C$ to $C \# C_H$ using the disk $D$.

Note that the regular neighborhood of the curve $C\#_b C_H$ is a Hopf annulus. We claim that there is an isotopy -- 
say $\Phi_t$ --
of  $\mathcal{D}E(m)$ which is fixed near the boundary of $\mathcal{D}E(m)$ and which induces a Dehn twist along
$C \#_b C_H.$ In fact, the isotopy can be assumed to be the identity when restricted to the $2$--handle as well.
This can be done as follows:

To begin with,  recall that the whole surface $\Sigma$ except the $2$-disk $D$ coming for the attached $2$--handle is still 
embedded in $B^4.$ In fact, we would like to point out that 
everything except the cylinders $\partial D_i \times [\frac{1}{2},1]$ are still embedded in the level
$S^3 \times \{\frac{1}{2}\}$ of the collar $S^3 \times [0,1]$ of $\partial B^4.$ In particular, a fixed neighborhood
$\mathcal{N}(C \#_b C_H)$ is contained in $S^3 \times \{\frac{1}{2}\}.$  

In order to get the isotopy $\Psi_t$ as
claimed we first describe how to produce an isotopy $\Phi_t$ of $\mathcal{D}E(m)$ which
induces the Dehn twist along $C \#_b C_H$ on $\Sigma.$ 

This is done as follows: Push the neighborhood $\mathcal{N}(C \#_b C_H)$ slightly towards $S^3 \times \{0\}$ 
in the collar in such a way
that at a fixed level between $0$ and $\frac{1}{2}$ the intersection of this pushed neighborhood  is a Hopf annulus and this
Hopf annulus contains the pushed  curve $C \#_b C_H$ as its center curve. Let us denote
this level by $S^3\times \{s_0\}.$  We now perform an isotopy to induce a Dehn twist along the pushed $C \#_b C_H$ in such 
a way that this
isotopy is supported in a small neighborhood of $S^3 \times \{s_0\}$ not intersecting $S^3 \times\{ \frac{1}{2}\}.$ After 
performing this isotopy, we further isotope the pushed neighborhood $\mathcal{N} (C \# C_H)$ back to its original
place in $S^3 \times \{\frac{1}{2}\}.$ Clearly, the effect of successive compositions of these isotopies is an isotopy
$\Phi_t$ which induces the Dehn twist along $C \# C_H$ on $\Sigma.$

We are almost done.  We now recall that the mapping class group of $(\Sigma, \partial \Sigma)$ is 
generated by Dehn twists along Lickorish curves as described in the Figure~\ref{fig:lic_gen} for an orientable
surface with one boundary component and as described in \cite[p. ~133]{FM} for an orientable surface with
more than one boundary components.
Since on each Lickorish curve it is possible to perform a Dehn twist via an ambient isotopy of $\mathcal{D}E(m)$,
we get the isotopy $\Psi_t$ with the required properties.

\end{proof}

\begin{remark}
 Notice that the Lemma~\ref{lem:flexible_embedding} above shows that in $\mathcal{D}E(m),$ there exists a proper flexible embedding of a the surface $\Sigma$. 
\end{remark}

\begin{proof}[Proof of Theorem~\ref{thm:smooth_ob_embedding}]

Consider the abstract open book $\mathcal{A}ob \left(\mathcal{D}E(m), Id\right)$. Recall from \cite{Ko} that if $m$ is even,
then $\mathcal{A}ob \left(\mathcal{D}E(m), Id\right)$  represents the manifold $S^3\times S^2$ and if $m$ is odd 
then it represents $S^2 \widetilde{\times} S^3$.

Observe that the Lemma~\ref{lem:flexible_embedding} implies that there is an  abstract
open book embedding of $\mathcal{A}ob (\Sigma, \phi)$ in $\mathcal{A}ob(\mathcal{D}E(m),Id)$,  
for any $n$, where $\mathcal{A}ob (\Sigma, \phi)$ is an abstract open book associated to the open 
book $\mathcal{O}b(B, \pi)$ of $M$. Hence, $(M, \mathcal{O}b(B, \pi))$ admits an open book embedding in 
any open book decomposition associated to $S^3 \times S^2$ and $S^2 \widetilde{\times} S^3$ 
with page a disk bundle over $S^2$ and monodromy the identity.

\end{proof}

\section{Embeddings of $3$--manifolds in $S^5$}

In this section, we use techniques developed to establish the Lemma~\ref{lem:flexible_embedding} 
 to reprove the theorem that every closed orientable $3$--manifold
embeds in $S^5.$

\begin{figure}[ht]
\begin{center}
\psfrag{K}{$K$}
\psfrag{+1}{$+1$}
\psfrag{K'}{$K'$}
\psfrag{U}{$U$}
\psfrag{L}{$K$}
\includegraphics[width=7cm,height=4cm]{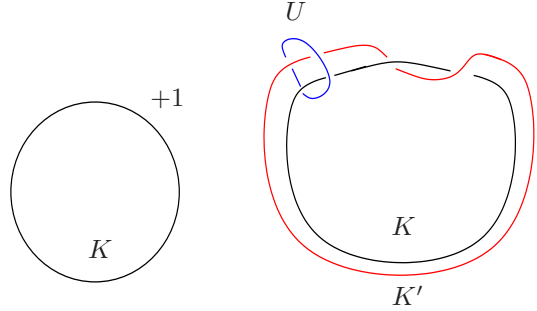}
\caption{In the left side of the above figure, we  depict the Kirby diagram of $\mathcal{D}E(1).$ The unknot $K$ with 
  framing $+1$ is the attaching circle for the $2$-handle of $\mathcal{D}E(1).$ While in the right side, 
  we depict the unknot $K$ together with
 the unknot $K'$ which is the boundary of a slightly pushed copy of the core of the attaching handle. 
 The blue knot is the unknot
 $U$ linking both $K$ and $K'$ once. The knot $K'$ is assumed to be on the boundary of the attaching region which is
 a solid torus around $K.$}
  \label{fig:kirby_diagram_DE_1}
  \end{center}
  \end{figure}

\begin{proof}[Proof of Theorem~\ref{thm:smooth_embedding_in_S^5}]

In order to produce an embedding of a closed orientable $3$--manifold $M$ in $S^5,$ we first notice that 
it is sufficient to embed $M$ in $S^3 \times \mathbb{R}^2.$ Hence, in what follows, we show how to embed $M$ in 
$S^3 \times \mathbb{R}^2.$

To begin with, we fix and review some notations. We parametrize a collar of $\partial B^4$ by $S^3 \times [0,1]$ such that 
$\partial B^4 = S^3 \times \{1\}.$
The unknot $K$ which is the attaching circle of the $2$-handle is then contained  in $S^3 \times \{1\}.$ This is 
depicted in the left of the Figure~\ref{fig:kirby_diagram_DE_1} by black circle with framing $+1$. 
Let us denote the zero section  of  the bundle $\mathcal{D}E(1)$ by $\mathcal{S}.$ We can 
regard $\mathcal{S}$ as
the sphere  obtained by considering the union of attaching  disk of the $2$-handle  with $K \times [0,1]$
and  the obvious disk $K \times \{0\}$ bounds in $S^3 \times \{0\}.$ 
Next, we denote by $\mathcal{N}(K)$ a tubular neighborhood of $K$ which is the attaching region of the $2$--handle 
$H_2 = D^2 \times D^2.$ If $p$ is a point on the boundary of $D^2,$ then the disk $D^2 \times \{p\}$ embedded
in the $2$--handle $H_2$  intersects the boundary $S^3 \times \{1\}$ in a curve $K'$ which links $K$ once. This is
depicted by the red curve in the Figure~\ref{fig:kirby_diagram_DE_1}. Notice that $K'$ lies on the boundary of
the attaching region $\mathcal{N}(K).$

We now describe how to embed a surface with one boundary component which is disjoint from the zero section $\mathcal{S}$ of
$\mathcal{D}E(1)$ and is flexible in $\mathcal{D}E(1).$ Consider an unknot $U$ which links the attaching region 
$\mathcal{N}(K)$ as depicted in the right of  Figure~\ref{fig:kirby_diagram_DE_1}. Consider the two 
circles $U \times \{\frac{1}{2}\}$ and $K' \times \{\frac{1}{2}\}$ in the sphere $S^3 \times \{\frac{1}{2}\}.$
Notice that the complement of $\mathcal{N}(K) \times \{ \frac{1}{2}\}$ in $S^3 \times \{\frac{1}{2}\}$ is a solid
torus $S^1 \times D^2$. The circle $U \times \{\frac{1}{2}\}$ is the center circle  $S^1 \times \{0\}$ of this 
solid torus while $K' \times \{\frac{1}{2}\}$ is a curve going once around the longitude and once around the
meridian of the solid torus. This implies circles $U \times \{\frac{1}{2}\}$ and $K' \times \{\frac{1}{2}\}$
bound a Hopf annulus in $S^3 \times \{\frac{1}{2}\}$ which is disjoint for $K \times \{\frac{1}{2}\}$ as it
lies inside the solid torus $S^1 \times D^2.$ Let us call
this Hopf annulus $\mathcal{A}.$ Now, observe  that  the boundary component of the annulus $\mathcal{A}$ corresponding
to $K' \times \{ \frac{1}{2}\}$ bounds a disk --  say $\mathcal{D}$ -- in $\mathcal{D}E(1)$ by construction. 
By attaching the annulus $U \times [\frac{1}{2},1]$ to $\mathcal{D}$ along its boundary
$U \times \{ \frac{1}{2}\},$ we produce a properly embedded disk in $\mathcal{D}E(1).$

Next, let  $\widetilde{\Sigma}$ be a standardly embedded  handle-body contained in the solid torus
$S^1 \times D^2$ which is disjoint from the Hopf annulus  $\mathcal{A}.$ Let $\Sigma$ be the boundary of 
this handle-body. We can perform an ambient connected sum of $\Sigma$ with $\mathcal{A}$ in 
$S^3 \times \{\frac{1}{2}\}$ such that the surface obtained after the ambient connected sum 
is still contained in the complement of $K \times \{\frac{1}{2}\}$ in $S^3 \times \{\frac{1}{2}\}.$ 
Notice that since $\mathcal{A}$ is an 
annulus in a properly embedded disk described in the previous paragraph, 
this connected sum operation produces a properly embedded surface with one boundary component. 
By a slight abuse of notation, let us continue to denote this surface by $\Sigma.$

Observe that an argument similar to the one used in the proof of the Lemma~\ref{lem:flexible_embedding} implies
that $\Sigma$ is a properly embedded flexible surface in $\mathcal{D}E(1).$ This is because the embedded 
surface $\Sigma$ admits an embedding of a Hopf annlus such that one of the boundary component of this Hopf
annlus bounds a disk in the surface by the construction. Hence,
we can isotope every generator of the mapping class group of $\Sigma$ in such a way that it
admits a neighborhood which is a Hopf annulus embedded in $S^3 \times \{\frac{1}{2}\}.$ Furthermore, notice that
$\Sigma$ does not intersect the zero section $\mathcal{S}$ of the bundle $\mathcal{D}E(1)$ as it does not intersect
the core disk of the $2$--handle as well as the annulus $K \times [\frac{1}{2},1].$

Now, let $M$ be any closed orientable $3$--manifold. Observe that it was
established in \cite{My} that we can regard $M$ as $\mathcal{A}ob(\Sigma, \phi)$ for 
some orientable surface $\Sigma$ with one boundary component. Since $\Sigma$ admits a flexible embedding
in $\mathcal{D}E(1),$  there exists an open book embedding of $M$ in 
$S^2 \widetilde{\times} S^3 = \mathcal{A}ob(\mathcal{D}E(1), Id).$ 

We now  notice that since 
$\Sigma$ does not intersect the zero section $\mathcal{S}$,  the mapping torus $\mathcal{M}T(\Sigma, \phi)$ 
associated to the abstract open book $M = \mathcal{A}ob(\Sigma, \phi)$ is in fact,
properly embedded  in a manifold diffeomorphic to $S^1 \times S^3 \times (0,1].$ 
This follows from
the fact that the complement of the zero section in $\mathcal{D}E(1)$ is $S^3 \times (0,1],$ see,
for example, \cite[p.~119]{GS}.

Next, consider the disjoint union of $S^1 \times S^3 \times (0,1]$ and $S^3 \times D^2.$ Consider the
quotient manifold obtained by identifying  the boundary  $S^1 \times S^3 \times \{1\}$ of
$S^1 \times S^3 \times (0,1]$ with
the boundary $S^3 \times S^1$ of $S^3 \times D^2$ by the 
identity. Notice that the resulting quotien manifold is diffeomorphic to $S^3 \times D^2.$


%

Notice that since the mapping torus $\mathcal{M}T(\Sigma, \phi)$ is properly embedded in
$S^1 \times S^3 \times (0,1],$ we clearly get that the embedding  of $M$ in $S^2 \widetilde{\times} S^3$
obtained via the open book embedding of $M$ in $S^2 \widetilde{\times} S^3 = \mathcal{A}ob(\mathcal{D}E(1), Id)$
is contained in a manifold diffeomorphic to  $S^3 \times \mathbb{R}^2$ as required. This completes our argument.
\end{proof}


\begin{thebibliography}{xxxx}
\bibitem[Al]{Al} J. Alexander, A lemma on systems of knotted curves, {\it Proc. Nat. Acad. Sci.}, vol. 9, (1923), 93--95.
 
 

\bibitem[DK]{DK} S. Durst , M. Klukas, Nested Open books and the binding sum, {\it arxiv:1610.07356v2[math.GT]}, (2017). 

 
 
 













\bibitem[EL]{EL} J. Etnyre  and Y. Lekili, Embedding all contact $3$--manifolds in a fixed contact $5$--manifold, 
{\it arXiv:1712.09642v1 [math.GT]}.



\bibitem[Et]{Et}  J. Etnyre, Lectures on open book decompositions and contact structures, 
{\it Floer homology, gauge theory, and low-dimensional topology} 5,  103--141.



\bibitem[FM]{FM} B. Farb and D. Margalit, A primer on mapping class groups, {\it Princeton Mathematical series},
vol. 49, Princeton University Press, (2012).

\bibitem[Ga]{Ga} D. Gabai, The Murasugi sum is a natural geometric operation. 
{\it Low-dimensional topology (San Francisco, Calif., 1981)},  
Contemp. Math., 20, Amer. Math. Soc., Providence, RI, 1983. 57M25 (57N10), 131--143.


\bibitem[Gi]{Gi} E. Giroux, G\'eom\'etrie de contact: de la dimension trois vers les dimensions sup\'erieures, 
{\it Proceedings of the ICM, Beijing} 2002, vol. 2, 405--414.






\bibitem[GS]{GS} R. Gompf and A. Stipsticz,  $4$--manifolds and Kirby Calculus. {\it Graduate studies in Mathematics},
vol. 20, AMS.










\bibitem[Hi]{Hi}    M. Hirsch,  The imbedding of bounding manifolds in euclidean space.
{\it Ann. of Math.} vol. 74 (3), (1961), 494--497.



\bibitem[HLM]{HLM} H. M. Hilden, M.T. Lozano and J.M. Montesinos, 
All three-manifolds are pullbacks of a branched covering
              $S^{3}$ to $S^{3}$, {\it Trans. Amer. Math. Soc.}, vol. 279  (2), (1983), 729--735.


              
\bibitem[HH]{HH}     A. Haefliger, M. Hirsch,  On the existence and classification of differentiable 
               embeddings, {\it Topology}, vol. 2, (1963),  129--135.
              


\bibitem[HY]{HY}
S. Hirose  and A. Yasuhara,  Surfaces in 4-manifolds and their mapping class groups,
   {\it Topology}, vol. 47 (1), 2008, 41--50.



\bibitem[Jo]{Jo} D. Johnson,
The structure of the Torelli group I: A finite set of generators for
$\mathcal{I}$, {\it Annals of Math.}, vol. 118 (2), (1983), 423--442.   

\bibitem[Ka]{Ka} S. Kaplan, Constructing framed $4$-manifolds with given almost framed boundaries, 
{\it Trans. Amer. Math. Soc.}, vol. 254, (1979), 237--263.




\bibitem[Ma]{Ma} D. Marti\'nez-Torres, Contact embeddings in standard contact spheres via
              approximately holomorphic geometry,
    {\it J. Math. Sci. Univ. Tokyo},
    vol. 18 (2), (2011),  139--154.




\bibitem[Mr]{Mr} A. Mori, Global models of contact forms.
{\it J. Math. Sci. Univ. Tokyo}, vol. 11 (4), 447--454, (2004).


\bibitem[My]{My} R. Myers, Open book decompositions of $3$--manifolds, {\it Proceedings of the AMS}, vol.72 (2),
(1978), 397--402.



\bibitem[La]{La} T. Lawson, Open book decomposition for odd dimensional manifolds, {\it Topology}, vol 17, (1979), 189--192. 



\bibitem[Li]{Li} W. Lickorish, A representation of orientable combinatorial $3$--manifolds,
  {\it Ann. of Math.}, vol. 76, (1962), 531--540,




















\bibitem[Ko]{Ko}  O. van Koert, Open books on contact five-manifolds, {\it Annales de l'Institut Fourier} vol. 58, 
(2008), 139--157.








\bibitem[Qu]{Qu} F. Quinn, Open book decompositions and the bordism of automorphisms, {\it Topology},
 vol. 18 (1), (1979), 55--73.


\bibitem[Ro]{Ro} V. Rohlin, The embedding of non-orientable three manifolds into 
five-dimensional Euclidean space (Russian). 
{\it Dokl. Akad. Nauk. SSSR}, vol. 160 (1965), 153-156. 
 
 
 


\bibitem[Ta]{Ta} I. Tamura, Spinnable structures on differentiable manifolds, {\it Proc. Japan Acad.} vol. 48,
(1972), 293--296.



\bibitem[Wa]{Wa} C. Wall, All $3$--manifolds imbed in $5$--space, {\it Bull. A.M.S (N.S)}, vol. 71, 
(1965), 564--567.


\bibitem[Wh]{Wh} H. Whitney, The self-intersections of a smooth $n$--manifold in $2n$--space, {\it Ann. of Math,}
vol. 45(2), (1944), 200--246.



\bibitem[Wi]{Wi} H. Winkelnkemper,  Manifolds as open books,
    {\it Bull. Amer. Math. Soc.}, vol. 7, (1973), 45--51.
		






\end{thebibliography}
\end{document}